\documentclass[11pt, reqno, a4paper]{amsart}

\usepackage[T1]{fontenc}
\usepackage{lmodern}          
\usepackage[margin=1.2in]{geometry}

\usepackage{amssymb}

\usepackage{enumitem}

\setlist[enumerate]{label=(\arabic*), font=\normalfont, leftmargin=*}

\usepackage{graphicx}
\usepackage{tikz}
\usetikzlibrary{decorations.pathreplacing, calligraphy, arrows.meta, positioning}

\usepackage{hyperref}
\hypersetup{
    colorlinks=true,
    linkcolor=blue!70!black,   
    citecolor=green!40!black,   
    urlcolor=blue!70!black,     
    pdfborder={0 0 0},          
    pdfauthor={Junhao Chen, Jie Zou},
    pdftitle={Borel Distinguishing Number of Finitely Generated Group Shifts}
}

\theoremstyle{plain}
\newtheorem{theorem}{Theorem}[section] 
\newtheorem{proposition}[theorem]{Proposition}
\newtheorem{lemma}[theorem]{Lemma}
\newtheorem{corollary}[theorem]{Corollary}
\newtheorem{question}[theorem]{Question}

\theoremstyle{definition}

\theoremstyle{remark}
\newtheorem{remark}[theorem]{Remark}

\newcommand{\Aut}{\operatorname{Aut}}   
\newcommand{\Cay}{\operatorname{Cay}}

\newcommand{\id}{\operatorname{id}}

\title{The Borel Distinguishing Number of Schreier Graphs}

\author{Junhao Chen}

\author{Jie Zou}

\date{\today}

\begin{document}

\begin{abstract}
The Borel distinguishing number $D_B(\mathcal{G})$ of a Borel graph $\mathcal{G}$, recently introduced by Bilge and Kaya, is the minimum number of colors required to break the symmetry of $\mathcal{G}$ in a Borel way. In this paper, we investigate the Borel distinguishing number of Schreier graphs induced by the free part of the shift action $\Gamma \curvearrowright n^\Gamma$. We prove that $D_B(\mathcal{G})\le n+1$ for $\Gamma=\mathbb{Z}^d$ equipped with the standard generators. Moreover, we show that $D_B(\mathcal{G})\ge n+1$ if $\Gamma$ is amenable and  $\{ \gamma \in \Aut(\Cay(\Gamma,S)) \mid \gamma(e) = e \}$ is non-trivial. We also  show that $D_B(\mathcal{G})$ is finite if $\Gamma$ is finitely generated, and give some applications of our results. These results answer some questions raised by Bilge and Kaya.
\end{abstract}

\maketitle

\section{Introduction}
Borel combinatorics is a new branch of descriptive set theory that studies classical combinatorial properties of graphs under certain ``definable'' requirements. In 1999, Kechris, Solecki, and Todorcevic  published a celebrated paper on Borel chromatic numbers \cite{KST99}, which is known as the birth of Borel combinatorics. In the following two decades, many researchers contributed to this new area. We refer the reader to the two surveys by Pikhurko \cite{Pik21} and Kechris and Marks \cite{KM20}. Borel combinatorics mainly focuses on Borel chromatic numbers, perfect matchings of graphs, and hyperfiniteness of connected equivalence relations. One reason why Borel combinatorics appeals to researchers is that answers to combinatorial questions might differ when definability requirements are imposed on the combinatorial objects. For example, the Schreier graph of an irrational rotation has chromatic number $2$, while its Borel chromatic number is $3$.

In \cite{BK25}, Bilge and Kaya introduced a new topic for Borel graphs, the Borel distinguishing number, which extends the notion in \cite{AC96} to characterize the symmetry of a Borel graph. In the classical setting, the distinguishing number of a finite connected graph is at most $\Delta(\mathcal{G})+1$, where $\Delta(\mathcal{G})$ denotes the maximum degree of the graph $\mathcal{G}$; see \cite{CT06}. Bilge and Kaya showed the differences between the classical and Borel settings of distinguishing numbers. They provided examples that separate the distinguishing number and the Borel distinguishing number at various levels. For instance, they proved that the Borel distinguishing number of a Schreier graph $\mathcal{G}$, induced by the free part of the shift action $\mathbb{Z} \curvearrowright n^\mathbb{Z}$, has an upper bound of $2n-1$ and a lower bound of $n$ for $n\ge 3$. They also showed that the classical distinguishing number of $\mathcal{G}$ is $2$; hence the distinguishing number and the Borel distinguishing number could be entirely different. However, they did not give the exact value of the Borel distinguishing number of $\mathcal{G}$. The following is a general question posed in \cite{BK25}:

\begin{question}
Let $\mathcal{G}$ be the Schreier graph of a countable marked group with a free shift action $\Gamma \curvearrowright n^\Gamma$. What are the values of $D_B(\mathcal{G})$ for various $(\Gamma,S)$ and alphabets $n$? In particular, what are the values in the case of $\Gamma=\mathbb{Z}$ or $\Gamma=\mathbb{F}_2$ equipped with the standard generators?
\end{question}

Our first result gives a lower bound in the amenable case. For a group $\Gamma$, let $e$ denote its unit element. We let $A_e(\Gamma)=\{\gamma \in \Aut(\Cay(\Gamma,S)) \mid \gamma(e) = e \}.$ Then:

\begin{theorem}\label{12}
Suppose $\Gamma$ is a finitely generated amenable group and $n\ge 2$ is an integer. Let $\mathcal{G}$ be the Schreier graph induced by the free part of the shift action $\Gamma \curvearrowright n^\Gamma$. 
\begin{enumerate}
    \item If $A_{e}(\Gamma)$ is trivial, then $D_{B}(\mathcal{G})= n$.
    \item If $A_{e}(\Gamma)$ is non-trivial, then $D_{B}(\mathcal{G})\ge n+1$.
\end{enumerate}
\end{theorem}
Part of the proof idea appears in \cite[Proposition 6]{BK25} in the case of $\Gamma=\mathbb{Z}$, where the Variational Principle for dynamical systems arising from $\mathbb{Z}$-actions is used. To generalize their results to amenable groups, we need some amenable entropy theory, which was  developed by Ornstein and Weiss \cite{OW87}. For $(2)$, we give a short proof by combining the variational principle with a trivial but critical lemma, which asserts that there exists a unique invariant measure, namely the Bernoulli measure, that attains maximal entropy.

Then we turn to the case of $\Gamma=\mathbb{Z}^d$. Our next result gives an upper bound.
\begin{theorem}
Let $d\ge 1,n\ge 2$ be two integers, and let $\mathcal{G}$ be the Schreier graph induced by the free part of the shift action $\mathbb{Z}^d \curvearrowright n^{\mathbb{Z}^d}$. Then $D_{B}(\mathcal{G})\le n+1$.
\end{theorem}

We construct a universal pattern and put it on the middle layers of some cubes to break the symmetry in a Borel way. Combining the above two theorems, we calculate the exact value in the case $\Gamma=\mathbb{Z}^d$.

For the free group $\mathbb{F}_2$,  Bilge and Kaya initially hypothesized that the Borel distinguishing number for the $\mathbb{F}_2$ shift action could be uncountable. We give a short proof using a proper coloring of a power graph to show that the Borel distinguishing number is finite for any finitely generated group over a finite alphabet. 

\begin{theorem}
Let $\Gamma$ be a finitely generated group with a symmetric generating set $S$ of size $d$. Let $\mathcal{G}$ be the Schreier graph induced by the free part of the shift action $\Gamma \curvearrowright Y^\Gamma$, where $Y$ is a standard Borel space. Then $D_B(\mathcal{G}) \le |Y| \times (d^2+1)^{d+1}. $
\end{theorem}

These two theorems answer parts of the questions from Bilge and Kaya. As a corollary of the above theorems, we also resolve the countably infinite alphabet case for finitely generated amenable groups.

This paper is organized as follows. In Section 2, we introduce some preliminaries and fix notation. In Section 3, we  give an equivalent characterization of Borel distinguishing colorings of Schreier graphs and show some applications. In Section 4, we get the lower bound by using amenable entropy theory. In Section 5, we obtain the upper bound for Schreier graphs induced by $\mathbb{Z}^d$. Finally, in Section 6, we show a general upper bound for all finitely generated groups.

\section{Preliminaries}
\subsection{Classical graph theory and Borel combinatorics}
Let $\mathcal{G}=(X,E)$ be an (undirected) graph, where $X$ is a set and $E$ is a symmetric and irreflexive subset of $X\times X$. $X$ is called the \textit{vertex set} and $E$ is called the \textit{edge set}. The graph $\mathcal{G}$ is \textit{locally finite} if the following holds: $\forall x\in X$, $\{y\in X \mid yEx\}$ is a finite set. The \textit{maximal degree} of $\mathcal{G}$, which we denote by $\Delta(\mathcal{G})$, is the supremum of $|\{y\in X \mid yEx\}|$ for all $x$.
A \textit{distinguishing $Y$-coloring} is a map $c: X \longrightarrow Y$ such that for any non-trivial graph automorphism $\psi$, the coloring is not preserved, i.e., $c \circ \psi \neq c$. In other words, a distinguishing coloring of $\mathcal{G}$ is a labeling of its vertices that results in a labeled graph without non-trivial symmetries. The \textit{distinguishing number} of a graph $\mathcal{G}$ is defined to be:
\[ D(\mathcal{G})=\min\{|Y| \mid \text{There exists a distinguishing coloring } c: X \longrightarrow Y \} \]

We introduce some basic definitions in descriptive set theory, and our reference is \cite{Kec95}. A \textit{standard Borel space} is a measurable space $(X,\mathcal{B})$ where $\mathcal{B}$ is the Borel $\sigma$-algebra induced by some Polish topology on $X$. Equivalently, it is a space equipped with the $\sigma$-algebra generated by the open subsets of some separable completely metrizable topological space. A \textit{Borel graph} $\mathcal{G}=(X,E)$ is a graph where the vertex set $X$ is a standard Borel space, and the edge set $E$ is a Borel subset of $X \times X$ which is equipped with the product $\sigma$-algebra. A \textit{Borel distinguishing $Y$-coloring} of $\mathcal{G}$ is a distinguishing coloring that is Borel measurable. The \textit{Borel distinguishing number} is defined to be:
\[ D_B(\mathcal{G})=\min\{|Y| \mid \text{There exists a Borel distinguishing coloring } c: X \rightarrow Y,\text{Y is standard Borel}\} \]

 For a Borel graph $\mathcal{G}=(X,E)$ and a standard Borel space $Y$, a \textit{Borel proper $Y$-coloring} of $\mathcal{G}$ is a Borel map $c: X \longrightarrow Y$ such that for each $x,y\in X$, $xEy \implies c(x)\ne c(y)$. The Borel chromatic number of $\mathcal{G}$ is defined to be:
\[ \chi_B(\mathcal{G})=\min\{|Y| \mid \text{There exists a Borel proper coloring } c: X \longrightarrow Y\} \]

Kechris, Solecki, and Todorcevic proved the following foundational theorem for graphs of bounded degree in \cite[Proposition~4.6]{KST99}:

\begin{theorem}
\label{KST}
For any Borel graph $\mathcal{G}$ with bounded maximum degree $\Delta(\mathcal{G})$, its Borel chromatic number satisfies $\chi_B(\mathcal{G})\le \Delta(\mathcal{G}) + 1$.    
\end{theorem}

Let $k$ be an integer. We consider the $k$-\textit{power graph} of $\mathcal{G}=(X,E)$, denoted by $\mathcal{G}^k$. The vertex set of $\mathcal{G}^k$ is still $X$, and for each $x,y\in X$, $xE(\mathcal{G}^k)y \iff 0<d_{\mathcal{G}}(x,y)\le k$, where $d_{\mathcal{G}}$ is the path metric of $\mathcal{G}$. By \cite[Corollary 5.2]{Pik21}, if $\mathcal{G}$ is a locally finite Borel graph, so is $\mathcal{G}^k$.

\subsection{Schreier graph }
Next, we turn to some algebraic and dynamical notation. In this paper, we always assume that every group is countably infinite. 
 Let $(\Gamma,S)$ be a discrete finitely generated group, where $S$ is a finite symmetric generating set not containing the identity. We let $e$ denote the unit element of the group. Suppose $\Gamma$ has a Borel action on a standard Borel space $X$. We define the \textit{Schreier graph} $\mathcal{G}= (X,E)$ on $X$ by:
$ x E y \iff x\ne y \wedge\exists s\in S, s\cdot x=y. $ We mainly consider the Schreier graph induced by the free part; that is, we restrict the vertex set to $\{x\in X:g\cdot x\ne x, \forall g\in\Gamma\setminus\{e\}\}$. By \cite[Lemma 4.4]{Pik21}, this is a Borel graph. We usually take this action to be the \textit{(right) shift}, i.e.:
$ (g \cdot x)(h) = x(hg)$.

For a Borel map $f: X \longrightarrow Y$, we define the \textit{symbolic factor map} $\varphi_f: X \longrightarrow Y^\Gamma$ by:
\[ \varphi_f(x)(g) = f(g \cdot x), \quad \forall g \in \Gamma. \]
Note that this map is  $\Gamma$-equivariant: 
\[ \varphi_f(s \cdot x)(g)=f(g \cdot (s\cdot x))=f((gs)\cdot x)=\varphi_f(x)(gs)=(s \cdot \varphi_f(x))(g). \]

We let $\Cay(\Gamma, S)$ denote the  Cayley graph of $\Gamma$ with respect to the symmetric generating set $S$; that is, for any two distinct elements $g,h\in \Gamma$, $gE(\Gamma)h\iff \exists s\in S, g=sh$. Let $\Aut(\Cay(\Gamma, S))$ be the group consisting of all graph automorphisms of $\Cay(\Gamma, S)$. Let $A_e(\Gamma,S)= \{ \gamma \in \Aut(\Cay(\Gamma, S))\mid \gamma(e)=e\}$. When there is no ambiguity, we will omit $S$.

On each product space $Y^\Gamma$, $A_e(\Gamma)$ has a natural action defined by:
\[ (\gamma \cdot \alpha)(g) = \alpha(\gamma^{-1}(g)), \quad \forall g \in \Gamma, \forall \gamma \in A_e(\Gamma). \]

\section{Distinguishing Coloring for Schreier Graphs}

We first state an equivalent condition for a Borel distinguishing coloring on Schreier graphs. This proposition generalizes \cite[Proposition 5]{BK25} and plays a central role in our proofs. In this section, we assume $\Gamma$ is a countable discrete group.

\begin{proposition}\label{equivalence}
Suppose $\Gamma$ acts freely on a standard Borel space $X$ and $Y$ is a standard Borel space. Let $\mathcal{G}=(X,E)$ be the induced Schreier graph, and let $f: X\longrightarrow Y$ be a Borel map. The following are equivalent:
\begin{enumerate}
    \item $f$ is a Borel distinguishing coloring of $\mathcal{G}$.
    \item $\varphi_f$ is injective, and for any $\gamma \in A_e(\Gamma) \setminus \{\id\}$ and any $x,y\in X$,  $\varphi_f(x) \ne \gamma\cdot\varphi_f(y)$.
    \item $\varphi_f$ is injective, and the family $\{\gamma\cdot \varphi_f(X)\}_{\gamma\in A_e(\Gamma)}$ is pairwise disjoint.
\end{enumerate}
\end{proposition}

\begin{proof}
$(2) \iff (3)$ is obvious.

$(1)\implies (2)$: We prove it by  contradiction. Suppose either $\varphi_f$ is not injective  or there exist $\gamma \in A_e(\Gamma)\setminus\{\id\}$ and $x,y \in X$ satisfying $\varphi_f(x) = \gamma\cdot\varphi_f(y)$. In the first case, 
we have $\varphi_f(x) =\varphi_f(y)$ for some distinct $x, y \in X$. If $x,y$ are in the same orbit, we can define an automorphism $\sigma$ by letting $\sigma (g\cdot x)=g\cdot y$ for each $g\in \Gamma$ and fixing all other orbits. If $x,y$ are not in the same orbit, we swap these two orbits by letting $\sigma (g\cdot x)=g\cdot y$ and $\sigma (g\cdot y)=g\cdot x$.

In the second case, for each $g\in \Gamma$, we have
$\varphi_f(x)(g)=(\gamma\cdot\varphi_f(y))(g)=\varphi_f(y)(\gamma^{-1}(g)).$
By the definition of $\varphi_f$, this means $f(g\cdot x)=f(\gamma^{-1}(g)\cdot y)$ for every $g\in\Gamma$. If $x,y$ are in the same orbit, define $\sigma$ on this orbit by
$\sigma(g\cdot x)=\gamma^{-1}(g)\cdot y,$
and fix all other orbits. If $x,y$ are not in the same orbit, define $\sigma$ on these two orbits by
$\sigma(g\cdot x)=\gamma^{-1}(g)\cdot y,\ \sigma(\gamma^{-1}(g)\cdot y)=g\cdot x$,
and fix all other orbits.

Such maps are non-trivial graph automorphisms of $\mathcal{G}$ and preserve the coloring $f$, contradicting the assumption that $f$ is distinguishing.

$(2)\implies(1)$: Suppose there is a $\psi\in \Aut(\mathcal G)\setminus\{\id\}$ satisfying $f\circ\psi=f$. Since $\psi$ is non-trivial, there is an $x\in X$ such that $x\ne\psi(x)$. Let $y=\psi(x)$.

For each fixed $g\in\Gamma$, by the assumption $f\circ\psi=f$, we have $\varphi_f(x)(g)=f(g\cdot x)=f(\psi(g\cdot x))$. Since the action is free and $\psi(g\cdot x)$ belongs to the orbit of $y=\psi(x)$, there is a unique element $h\in\Gamma$ such that $\psi(g\cdot x)=h\cdot y$. We denote $h$ by $\gamma(g)$.

We claim that $\gamma\in A_e(\Gamma)$. Since $\psi$ maps the orbit of $x$ bijectively onto the orbit of $y$ and the action is free, $\gamma$ is a bijection. Suppose that $g$ and $sg$ are adjacent in the  Cayley graph, where $s\in S$. Then $g\cdot x$ and $(sg)\cdot x=s\cdot(g\cdot x)$ are adjacent in $\mathcal G$. Hence $\gamma(g)\cdot y$ and $\gamma(sg)\cdot y$ are adjacent. By freeness, $\gamma(g)$ and $\gamma(sg)$ are adjacent in the  Cayley graph. By a similar argument, $\gamma^{-1}$ is a graph homomorphism. Moreover, $\gamma(e)=e$, because $\psi(x)=\psi(e\cdot x)=\gamma(e)\cdot\psi(x)$ and the action is free. Hence $\gamma\in A_e(\Gamma)$.

Therefore, for every $g\in\Gamma$,
$\varphi_f(x)(g)=f(g\cdot x)=f(\psi(g\cdot x))=f(\gamma(g)\cdot\psi(x))=\varphi_f(\psi(x))(\gamma(g))=(\gamma^{-1}\cdot\varphi_f(\psi(x)))(g)$. Thus $\varphi_f(x)=\gamma^{-1}\cdot\varphi_f(\psi(x))$.

If $\gamma=\id$, then $\varphi_f(x)=\varphi_f(\psi(x))$, and since $\varphi_f$ is injective, we get $x=\psi(x)$, contradicting the choice of $x$. If $\gamma\ne\id$, then $\gamma^{-1}\in A_e(\Gamma)\setminus\{\id\}$, and the equality $\varphi_f(x)=\gamma^{-1}\cdot\varphi_f(\psi(x))$ contradicts condition $(2)$. Hence $f$ is a distinguishing coloring.
\end{proof}

\begin{remark}\label{11}
By the proposition above, if $f$ is a Borel distinguishing $Y$-coloring, then the symbolic factor map $\varphi_f$ is a Borel equivariant embedding into $Y^\Gamma$. Also, to find a Borel distinguishing $Y$-coloring $f$, it suffices to find a Borel equivariant injection $\varphi: X \longrightarrow Y^\Gamma$ with the following property: for any $\gamma \in A_e(\Gamma) \setminus \{\id\}$ and $x,y\in X$, $\varphi(x) \ne \gamma\cdot\varphi(y)$, since $f(x)=\varphi(x)(e)$ will be the desired Borel distinguishing $Y$-coloring. 
\end{remark}

As  direct applications of Proposition \ref{equivalence}, we have the following two corollaries.

\begin{corollary}
Let $\mathcal{G}=(X,E)$ be the Schreier graph induced by the free part of the shift action $\Gamma \curvearrowright n^\Gamma$. If  $A_e(\Gamma)$ is trivial, then $D_B(\mathcal{G}) \le n$.
\end{corollary}
\begin{proof}
Let $c: X\longrightarrow n$ be defined by $c(x) = x(e)$.  The induced symbolic factor map $\varphi_c$ satisfies $\varphi_c(x)(g) = c(g\cdot x)=(g\cdot x)(e)=x(g)$ for each $g$. Thus $\varphi_c$ is  injective. Since $A_e(\Gamma) = \{\id\}$, Condition (2) in Proposition \ref{equivalence} is  satisfied. Consequently, $D_B(\mathcal{G}) \le n$.
\end{proof}

\begin{corollary}\label{embedding_comparison}
Let $\mathcal{G}$ and $\mathcal{H}$ be two Schreier graphs induced by two free shift actions $\Gamma \curvearrowright X$ and $\Gamma \curvearrowright Y$, respectively. If there is a Borel equivariant injection $f$ from $X$ to $Y$, then we have $D_B(\mathcal{G})\le D_B(\mathcal{H})$. 
\end{corollary}
\begin{proof}
Suppose $c:Y \longrightarrow k$ is a Borel distinguishing coloring for $\mathcal{H}$. The map $c$ induces a symbolic factor map $\phi_c :Y \longrightarrow k^\Gamma$. Then $\sigma=\phi_c \circ f$ is an injective equivariant Borel map from $X$ to $k^\Gamma$. Since $c$ is a distinguishing coloring for $\mathcal{H}$, by Proposition \ref{equivalence}, the family of sets $\{\gamma\cdot \phi_c(Y)\}_{\gamma\in A_e(\Gamma)}$ is pairwise disjoint.
Given that $f(X) \subseteq Y$, we have $\sigma(X) \subseteq \phi_c(Y)$. For any distinct $\gamma_1, \gamma_2 \in A_e(\Gamma)$, we have:
\[ \gamma_1\cdot \sigma(X) \cap \gamma_2\cdot \sigma(X) \subseteq \gamma_1\cdot \phi_c(Y) \cap \gamma_2\cdot \phi_c(Y) = \emptyset. \]
Thus, $\sigma$ satisfies condition $(3)$ of Proposition \ref{equivalence} for $\mathcal{G}$, which implies $D_B(\mathcal{G})\le D_B(\mathcal{H})$.
\end{proof}

\section{The Lower Bound for Amenable Group Shifts}
 
In this section, we begin to prove Theorem \ref{12}. Recall that a countable discrete group $\Gamma$ is said to be \textit{amenable} if it admits a  F{\o}lner sequence, that is, a sequence of non-empty finite subsets $\{F_n\}_{n\in\mathbb{N}}$ of $\Gamma$ such that for every  $
 g\in \Gamma$,
\begin{equation}
    \lim_{n\to\infty} \frac{|gF_n \triangle F_n|}{|F_n|} = 0.
\end{equation}
 
 Our tools come from the entropy theory for amenable groups developed by Ornstein and Weiss \cite{OW87}. Kerr and Li published an elegant textbook \cite{KL16} to give a complete introduction to amenable entropy theory; this textbook will be the main reference for this section.

 We only list some theorems and give a short introduction to measure entropy in order to prove Lemma \ref{uniqueness} and Theorem~\ref{lowerbound}.
  For more details on measure entropy and topological entropy of amenable group actions, see \cite[Definition 9.3 , Definition 9.28]{KL16}.
 
 We assume $\Gamma$ is a countably infinite amenable group. Let $(X, \mu)$ be a probability measure space, consider an action $\Gamma \curvearrowright (X, \mu)$, and let $a$ denote this action. The tuple $(\Gamma,X,\mu,a)$ is said to be a \textit{(measure) dynamical system}.
 We say this action is a measure-preserving action if $\mu(A)=\mu(g\cdot A)$ for every $g$ and every measurable set $A$. Equivalently, we say $\mu$ is $\Gamma$-invariant. Throughout this section, all actions  are assumed to be measure-preserving. When there is no ambiguity about the action, we usually omit $a$. 
 For a finite measurable partition $\mathcal{P}$ of $X$, let $H_\mu(\mathcal{P}) = -\sum_{A \in \mathcal{P}} \mu(A) \log \mu(A)$. For a finite measurable partition $\mathcal{P}$ of $X$ and a finite set $F\subseteq\Gamma$, we write $\mathcal{P}^F=\bigvee_{g\in F}g^{-1}\mathcal{P} $ for the join (see \cite[page 2]{KL16}) of the partitions $g^{-1}\mathcal{P}$, where $ g^{-1}\mathcal{P}=\{g^{-1}P:P\in\mathcal{P}\}. $

Let $(F_n)$ be a F{\o}lner sequence of $\Gamma$. We write $K\Subset\Gamma$ to mean that $K$ is a finite subset of $\Gamma$.
Define
    \[
    h_{\mu}(\Gamma, \mathcal{P}) = \lim_{n\to\infty} \frac{H_\mu(\mathcal{P}^{F_n})}{|F_n|} =\inf_{\varnothing\ne K\Subset\Gamma}\frac{H_\mu(\mathcal{P}^K)}{|K|}.
    \]
The limit exists and is independent of the choice of the F{\o}lner sequence. The \textit{measure entropy} of the dynamical system $(\Gamma,X,\mu,a)$ is defined as the supremum  over all finite measurable partitions:
    \[
    h_{\mu}(X) = \sup_{\mathcal{P}} h_{\mu}(\Gamma, \mathcal{P}).
    \]
From the definition, we can see that measure entropy is invariant under measure-theoretic isomorphism: if two actions are equivariantly isomorphic modulo null sets, then they have the same measure entropy.

\begin{theorem}[{\cite[Theorem 9.9]{KL16}}]
Let $\Gamma$ be a countable amenable group, and let $\lambda$ be the uniform Bernoulli measure on $n^\Gamma$. Suppose $\Gamma$ acts on $n^\Gamma$ by the shift action. Then the measure entropy is $\log n$.
\end{theorem}
\begin{remark}
Theorem 9.9 in \cite{KL16} states that the measure entropy equals the Shannon entropy $H(\nu)$, where $\nu$ is a probability measure on $n$. Here $\lambda$ is the product measure $\nu^\Gamma$. We can easily calculate the result by the definition of Shannon entropy in Section 9.1. Also note that we consider the free part of $\Gamma \curvearrowright n^\Gamma$, which is well known to be a conull set of $n^\Gamma$; hence restricting the underlying space to the free part does not change the measure entropy of this system. 
\end{remark}
 
\begin{theorem}[{\cite[Example 9.41]{KL16}}]
Let $\Gamma$ be a countable amenable group, and suppose $\Gamma$ acts on $n^\Gamma$ by the shift action. Then the topological entropy is $\log n$.
\end{theorem}

The connection between topological and measure entropy is established by the following theorem, which is known as the Variational Principle. 

\begin{theorem}[{\cite[Theorem 9.48]{KL16}}]\label{variational}
Let $\Gamma$ be a countable amenable group, and let $\Gamma \curvearrowright X$ be a continuous action on a compact metric space $X$. Then the topological entropy is the supremum of the measure entropy over all $\Gamma$-invariant probability measures.
\end{theorem}

We also need the following easy but critical fact. Since we have not found a reference, we give a proof here for completeness.

\begin{lemma}\label{uniqueness}
Let $\Gamma$ be an amenable group, and let $\Gamma \curvearrowright n^\Gamma$ be the shift action. Then the Bernoulli measure $\lambda$ is the unique $\Gamma$-invariant Borel probability measure whose measure entropy is $\log n$.
\end{lemma}
\begin{proof}
 Suppose for contradiction that there is a $\Gamma$-invariant Borel probability measure $\mu\ne\lambda$ such that $h_\mu(n^\Gamma)=\log n$. Since the Borel $\sigma$-algebra of $n^\Gamma$ is the product $\sigma$-algebra, there exists a cylinder set of the form $\{x\in n^\Gamma:x|_F=a\}$ such that $\mu(\{x\in n^\Gamma:x|_F=a\})\ne\lambda(\{x\in n^\Gamma:x|_F=a\})=n^{-|F|}$, where $F\subseteq\Gamma$ is finite,  $a\in n^F$, and $x|_F$ means the restriction of the map $x$ to $F$. Let $\mathcal{P}=\{P_0,\ldots,P_{n-1}\}$, where $P_i=\{x\in n^\Gamma:x(e)=i\}$. It is easy to see that $\mathcal{P}$ is a finite generating partition (see \cite[Definition 9.7]{KL16} ) of $n^\Gamma$. Then $\mathcal{P}^F=\bigvee_{g\in F}g^{-1}\mathcal{P}$ is the finite partition whose elements are cylinder sets $\{x\in n^\Gamma:x|_F=b\}$ for $b\in n^F$. Since one of them does not have measure $n^{-|F|}$, the elements of $\mathcal{P}^F$ do not all have the same measure. Combining this with \cite[Proposition 9.1 (2)]{KL16}, we have
\[
H_\mu(\mathcal{P}^F)<\log|\mathcal{P}^F|=\log(n^{|F|})=|F|\log n.
\]
By  \cite[Theorem 9.8]{KL16}, we have
\[
h_\mu(n^\Gamma)=h_{\mu}(\Gamma, \mathcal{P})=\inf_{\varnothing\ne K\Subset\Gamma}\frac{H_\mu(\mathcal{P}^K)}{|K|}\le\frac{H_\mu(\mathcal{P}^F)}{|F|}<\log n,
\]
which contradicts the assumption that $h_\mu(n^\Gamma)=\log n$. Hence $\mu=\lambda$, and the uniqueness follows. \end{proof}

Now we begin to prove the lower bound.

\begin{theorem}\label{lowerbound}
Let $\Gamma$ be a finitely generated amenable group and let $n\ge 2$ be an integer. Let $\mathcal{G}=(X,E)$ be the Schreier graph induced by the free part of the shift action $\Gamma \curvearrowright n^\Gamma$. Then:
\begin{enumerate}
    \item  $D_{B}(\mathcal{G})\ge n$.
    \item  If $A_{e}(\Gamma)$ is non-trivial, then $D_{B}(\mathcal{G}) \ge n+1$.
\end{enumerate}
\end{theorem}

\begin{proof}
Equip $n^\Gamma$ with the uniform Bernoulli measure $\lambda$. We equip $X$ with the restricted measure $\mu = \lambda|_X$, which has measure entropy $h_{\mu}(X)=\log n$. 

 Suppose for contradiction that $D_B(\mathcal{G}) \le n-1$. By Proposition \ref{equivalence}, there is a Borel equivariant injection $\varphi: X\longrightarrow (n-1)^\Gamma$. By the Luzin-Suslin theorem \cite[Theorem 15.1]{Kec95}, its image $A = \varphi(X)$ is a Borel set; hence $\varphi$ is a measure isomorphism between the systems $(X,\mu)$ and $(A,\nu)$, where $\nu = \varphi_* \mu$ is the pushforward measure. Since $A$ is a conull set of $(n-1)^{\Gamma}$ under $\nu$, $(X,\mu)$ and $((n-1)^{\Gamma},\nu)$ are measure-theoretically isomorphic. Therefore, $\nu$  achieves the  measure entropy $h_\nu((n-1)^\Gamma)=\log n$. 

However, by the variational principle (Theorem \ref{variational}), the measure entropy of any invariant probability measure on  $(n-1)^\Gamma$ is  bounded  above by its topological entropy, which is  $\log (n-1)$. It is easy to see that $\nu$ is an invariant measure on $(n-1)^{\Gamma}$; hence $\log n \le \log (n-1)$, a contradiction. Thus $D_B(\mathcal{G}) \ge n$.

Now we turn to the proof of (2). Suppose $A_e(\Gamma)$ is non-trivial. Assume for contradiction that $D_B(\mathcal{G}) \le n$. Then there exist a Borel distinguishing coloring $f:X\longrightarrow n$ and an induced Borel equivariant injection $\varphi_f:X\longrightarrow n^\Gamma$. By the same argument, $\nu = (\varphi_f)_*\mu$ on $n^\Gamma$  attains the measure entropy $h_\nu(n^\Gamma)=\log n$. By Lemma \ref{uniqueness},  $\nu=\lambda$. But $\nu$ is supported on the image $A=\varphi_f(X)$. Thus $\nu(A)=1\implies \lambda(A)=1$.
By assumption, there exists some non-trivial element $\gamma \in A_e(\Gamma) \setminus \{\id\}$. The Bernoulli measure $\lambda$ is  invariant under the $A_e(\Gamma)$ action, since each element in $A_e(\Gamma)$ is a bijection that only permutes coordinates, and the uniform Bernoulli  measure is
invariant under coordinate permutations. Thus, 
\[ \lambda(\gamma \cdot A)=\lambda(A)=1\implies\lambda (A\cap \gamma\cdot A)=1\implies A\cap\gamma\cdot A\ne \emptyset. \]
This contradicts condition (3)  in Proposition \ref{equivalence}.
\end{proof}

\section{Upper bound}
  Now we start to show the upper bound in the $\mathbb{Z}^d$ case for some integer $d\ge 1$. For any $d\ge 1$, a $d$-dimensional \textit{cube} $B$ in $\mathbb{Z}^d$ is a set of the form $B=[a_1,b_1]_\mathbb{Z}\times\dots\times [a_d,b_d]_\mathbb{Z}$, where $a_i<b_i$, $[a_i,b_i]_\mathbb{Z}$ is the integer interval,   and $b_1-a_1=\dots=b_d-a_d$ is the \textit{length} of $B$. For convenience, we only consider cubes with even length $k\ge 6$. The \textit{interior} of $B$, which we denote by $int(B)$, is $[a_1+1,b_1-1]_\mathbb{Z}\times\dots\times [a_d+1,b_d-1]_\mathbb{Z}$. We let $B\setminus int(B)$ be the boundary of $B$ and denote it by $\partial B$. A face of $B$ is a set of the form $\{x\in B:\pi_i(x)=a_i\}$ or $\{x\in B:\pi_i(x)=b_i\}$ for some $1\le i \le d$, where $\pi_i$ is the $i$-th projection map.
 Note that $int(B)$ is still a cube in $\mathbb{Z}^d$. 
 The \textit{core} of $B$ is the set $int(int(B))$, which we denote by $C(B)$. The \textit{middle layer} of $B$, which we denote by $M(B)$, is the set $int (B)\setminus C(B)$.

 For each positive even number $k$, there is a \textit{standard cube} of length $k$ for $\mathbb{Z}^d$ centered at $0$, that is, the cube $Q_k={[-\frac{k}{2},\frac{k}{2}]_\mathbb{Z}}^d$. We denote it by $Q_k$. A \textit{pattern} $p$ is defined as a map from $M(Q_k)$ to some integer $n$.  Every cube $B$ of length $k$ in $\mathbb{Z}^d$ can be written as $Q_k+g$ for a unique $g\in \mathbb{Z}^d$; we write it as $B_g$.  For a coloring $\alpha:\mathbb{Z}^d\longrightarrow n$, we say $\alpha$ has pattern $p$ on $B_g$ if $\alpha|_{M(B_g)}:M(B_g)\longrightarrow n$ satisfies $\alpha(h+g)=p(h)$ for every $h\in M(Q_k)$.

In this section, we assume that $\mathbb{Z}^d$ is equipped with the standard generators $S=\{\pm e_1,\ldots,\pm e_d\}$. Let $0$ denote the unit element of $\mathbb{Z}^d$. We first show a critical lemma.
\begin{lemma}\label{Ae}
For every 
$\gamma\in A_e(\mathbb{Z}^d)$, there exist a permutation $\sigma$ of $\{1,\dots,d\}$ and a vector $\varepsilon=(\varepsilon_1,\ldots,\varepsilon_d)\in\{-1,1\}^d$ such that for each $x\in \mathbb{Z}^d$, $\gamma(x) $ can be written as $\sum_{i=1}^d \varepsilon_i x_i e_{\sigma(i)}.$
\end{lemma} 
\begin{proof}
  We first prove the claim for generators. Since $\gamma$ fixes the unit element $0$, it must map a generator to a generator.  In fact, for each opposite pair of generators $\{e_i,-e_i\}$, it must be mapped to an opposite pair $\{e_j,-e_j\}$. This follows because $\gamma$ preserves shortest paths: $e_i,-e_i$ have only one shortest path connecting them, whereas $u$ and $v$ have two whenever they form a non-opposite pair $u,v\in S$. We let $\sigma(i)=j$. Now there is an $\varepsilon_i$ such that $\gamma(e_i)=\varepsilon_ie_{\sigma(i)}$. We do this for all generators and get $\varepsilon$ and a permutation $\sigma$. 
  
  Then we show that this holds for each element on a coordinate axis. For every  $n>0$, there is a unique shortest path from $0$ to $ne_i$, which is $(0,e_i,2e_i,\dots,ne_i)$. Since the first step of the image path to $\gamma (ne_i)$ is $\varepsilon_ie_{\sigma(i)}$, by induction, we must have $ \gamma(ne_i)=n\varepsilon_i e_{\sigma(i)}$. This also holds for $n<0$.

  Finally, for each $x=(x_1,\ldots,x_d)\in \mathbb{Z}^d$, let $\gamma(x)=y=(y_1,\dots,y_d)$ and choose an integer $N>d(x,0)$. Since $d(0,x)=d(0,y)$, $N>|y_j|$ for every $j$. Now we compare the distances to points on each coordinate axis. For each $i$, we have
$d(x,-Ne_i)-d(x,Ne_i)=2x_i$.
On the other hand, $2x_i=d(x,-Ne_i)-d(x,Ne_i)=d(y,-N\varepsilon_i e_{\sigma(i)}\bigr)
  -d(y,N\varepsilon_i e_{\sigma(i)}\bigr)=2\varepsilon_i y_{\sigma(i)}.$
Therefore, $y_{\sigma(i)}=\varepsilon_i x_i$ for every $i$. Since $\sigma$ is a permutation, $y
=\sum_{i=1}^d y_{\sigma(i)}e_{\sigma(i)}$. It follows that
$\gamma(x)=y
=\sum_{i=1}^d \varepsilon_i x_i e_{\sigma(i)}.$
\end{proof}
\begin{corollary}\label{cube}
Let $\gamma\in A_e(\mathbb{Z}^d)$, and let
$B_g=Q_k+g$ be a cube of even length $k$ centered at
$g\in\mathbb{Z}^d$. Then $\gamma(B_g)=Q_k+\gamma(g).$ Precisely, $\gamma(\partial B_g)=\partial\gamma(B_g),
\gamma(M(B_g))=M(\gamma(B_g)),\gamma(C(B_g))=C(\gamma(B_g)).$
Furthermore, for every $u\in Q_k$,
$\gamma(u+g)=\gamma(u)+\gamma(g).$
\end{corollary}

\begin{theorem}\label{upperbound}
Let $d\ge 1, n\ge 2$ be two integers, and let $\mathcal{G}=(X,E)$ be the Schreier graph induced by the free part of the shift action $\mathbb{Z}^d \curvearrowright n^{\mathbb{Z}^d}$. Then $D_{B}(\mathcal{G})\le n+1$.
\end{theorem}
\begin{proof}
Our proof relies on  Remark  \ref{11}. We will pointwise construct a Borel equivariant injection $\varphi: X\longrightarrow (n+1)^{\mathbb{Z}^d}$ such that $\varphi(\alpha)\ne\gamma\cdot\varphi(\beta)$ for all $\alpha,\beta\in X$ and all $\gamma\in A_e(\mathbb{Z}^d)\setminus\{\id\}$. 

We first choose a sufficiently large even number $k$ such that:
\begin{enumerate}
    \item $k>$ Max$\{6,3d+2\}$.
    \item $n^{(k+1)^d}\leq (n+1)^{(k-3)^d}$.
\end{enumerate}
We  show that there is a pattern $p$ such that, for
every  $\gamma\in A_e(\mathbb{Z}^d)\setminus\{id\}$, every
$\alpha\in (n+1)^{\mathbb{Z}^d}$, and every cube $B$ of length
$k$, if $\alpha$ has pattern $p$ on $B$, then
$\gamma\cdot\alpha$ does not have pattern $p$ on
$\gamma(B)$.
 
If $d\ge 2$, let
$v=(1,2,\dots,d-1,\frac{k-2}{2})\in M(Q_k)$. If $d=1$, let
$v=\frac{k-2}{2}$. We will show that $\gamma(v)\neq v$ for every 
$\gamma\in A_e(\mathbb{Z}^d)\setminus\{id\}$. Suppose $\gamma(v)=v$.
By Lemma \ref{Ae}, there exist $\sigma$ and
$\varepsilon\in\{-1,1\}^d$ such that $\gamma(v)=\sum_{i=1}^d \varepsilon_i v_i e_{\sigma(i)}.$ Hence $\varepsilon_i v_i=v_{\sigma(i)}$ for every $i$. Since the coordinates of $v$ are positive and pairwise distinct, we have $\varepsilon_i=1$ and $\sigma(i)=i$
for every $i$. Thus $\gamma=\id$, a contradiction. We color $v$ with $1$ and
all other points in $M(Q_k)$ with $0$. This is the pattern
$p$ we need.

Then we show that, for each cube $B$ of length $k$, if a
coloring has pattern $p$ on $B$, then, after applying any
 $\gamma\in A_e(\mathbb{Z}^d)\setminus\{id\}$, the induced pattern
on $\gamma(B)$ is not equal to $p$. Let $B=Q_k+g$. The unique
point colored with $1$ in $M(B)$ is $v+g$. By
Corollary \ref{cube},  the unique point colored with $1$ in  $M(\gamma(B))$ is $\gamma(v)+\gamma(g)$. Since
$\gamma(v)\neq v$ for every 
$\gamma\in A_e(\mathbb{Z}^d)\setminus\{id\}$, this  pattern is
different from $p$.

By our choice of $k$, there exists an injection $\psi: n^{Q_k} \longrightarrow (n+1)^{C(Q_k)}$ . We have also fixed a pattern $p$ of length $k$ for this $k$.

By Theorem \ref{KST}, there is a Borel proper coloring $c:X\longrightarrow l$ of the power graph $\mathcal{G}^{4dk}$, where $l\in\mathbb{N}$ is an integer. For each $\alpha\in X$, define the center set
\[
\mathcal{C}(\alpha)=\left\{g\in\mathbb{Z}^d:c(g\cdot\alpha)=\min_{h\in\mathbb{Z}^d}c(h\cdot\alpha)\right\}.
\]
We define $\varphi$ as follows: for each $g\in \mathcal{C}(\alpha)$, let $B_g$ denote the cube of length $k$ centered at $g$. We recolor this cube according to the following rules:
\begin{enumerate}
    \item Fill $\partial B_g$ with the new color $n$ to help us identify those changed cubes.
    \item On $M(B_g)$, we place the pattern $p$ to break the symmetry.
    \item For $C(B_g)$, we shift $B_g$ to the standard cube $Q_k$, encode the original coloring of $B_g$ into its core using $\psi$,
and then shift it back to its original position.
\end{enumerate}

We do this for all $g\in \mathcal{C}(\alpha)$ and leave all other areas with their original colors. Let the result be $\varphi(\alpha)$. 

Now we show that $\varphi$ satisfies our requirements. We only need to check these changed areas. Because $c$ and the action of ${\mathbb{Z}}^d$ are Borel and, for a fixed $x\in \mathbb{Z}^d$, determining whether $x$ lies in a changed area involves a countable verification, $\varphi$ is Borel. For the equivariance, note that for each $g\in\mathbb{Z}^d$, 
\[\forall x\in {\mathbb{Z}}^d, x\in \mathcal{C}(g\cdot\alpha)\iff c((x+g)\cdot\alpha)\text{ has the least color}\iff x+g\in \mathcal{C}(\alpha).\]
 Hence $\mathcal{C}(g\cdot\alpha)=\mathcal{C}(\alpha)-g$, which means all the centers move equivariantly, as do all the cubes. Moreover, the color pattern does not change under the shift action. This implies equivariance.

 For injectivity, since $c$ is a proper coloring, any two distinct centers are at a distance greater than $4dk$. So any two distinct changed cubes $B$, $B'$ are at a distance greater than $dk$. Then we can identify each changed cube through its boundary of length $k$. In fact, the color $n$ does not occur in the middle layers or outside these cubes. Thus, even if the color $n$ occurs inside a core, the core is not large enough to have length $k$, so we will not confuse the core with the boundary. After we identify the position of each changed cube,  we can check the core and restore the original coloring through the injection $\psi$.
 
Finally, we show that for all $\alpha,\beta\in X$ and all $\gamma\in A_e(\mathbb{Z}^d)\setminus\{\mathrm{id}\}$, $\varphi(\alpha)\ne\gamma\cdot\varphi(\beta)$. Suppose there are $\alpha,\beta\in X$ and $\gamma\in A_e(\mathbb{Z}^d)\setminus\{\mathrm{id}\}$ such that $\varphi(\alpha)=\gamma\cdot\varphi(\beta)$. We fix a changed cube $B$ in $\varphi(\beta)$. By our choice of $p$, $\gamma\cdot\varphi(\beta)$ has another pattern $p'$ on the image cube. Note that $\gamma(B)$ is a changed cube in $\varphi(\alpha)$. However, by rule (2), the pattern $p'$ will not occur on the middle layer of any changed cube, a contradiction.
\end{proof}
\begin{corollary}
Let $d\ge 1$ and $n\ge 2$ be two integers. Let $\mathcal{G}$ be the Schreier graph induced by the free part of the shift action $\mathbb{Z}^d \curvearrowright n^{\mathbb{Z}^d}$. Then $D_B(\mathcal{G}) = n+1$.
\end{corollary}
\begin{proof}
It is well known that $\mathbb{Z}^d$ equipped with the standard generating set is amenable. Moreover, since the map $g\longmapsto g^{-1}$ is an element of $A_e(\mathbb{Z}^d)$, it follows that $A_e$ is nontrivial. The conclusion follows by combining Theorem \ref{lowerbound} and Theorem \ref{upperbound}.
\end{proof}

\section{Upper Bounds for General Finitely Generated Groups}
In this section, we give an upper bound for all finitely generated group shifts.

\begin{theorem}\label{generalbound}
Let $\Gamma$ be a finitely generated group with a symmetric generating set $S$ of size $d$. Let $\mathcal{G}=(X,E)$ be the Schreier graph induced by the free part of the shift action $\Gamma \curvearrowright Y^\Gamma$, where $Y$ is a standard Borel space. Then $D_B(\mathcal{G}) \le |Y| \times (d^2+1)^{d+1}$. 
\end{theorem}

\begin{proof}
 Consider the $2$-power graph $\mathcal{G}^2$. The maximum degree of $\mathcal{G}^2$  satisfies $\Delta(\mathcal{G}^2) \le d + d(d-1) = d^2$. By Theorem \ref{KST}, we have $\chi_B(\mathcal{G}^2) \le \Delta(\mathcal{G}^2) + 1 \le d^2 + 1$. Hence, $\mathcal{G}^2$ has a Borel proper coloring $k: X \longrightarrow C$, where  $|C| \le d^2+1$. Define the coloring map $c: X \to Y \times C^{d+1}$ as follows:
\[
    c(x) = \left( x(e), k(x), k(s_1 \cdot x), \dots, k(s_d \cdot x) \right),
\]
where $S=\{s_1,\dots,s_d\}$.

We claim that $c$ is a Borel distinguishing coloring. Suppose there exists a graph automorphism $\varphi \in \Aut(\mathcal{G})$ such that $c \circ \varphi = c$. We will show that $\varphi = \id$. First we show that $\varphi$ is a $\Gamma$-equivariant map. Since $\varphi$ is a graph automorphism, for any given $s_i \in S$,  there exists some generator $s_j \in S$  such that $\varphi(s_i \cdot x) = s_j \cdot \varphi(x)$. It  suffices to show that $i=j$. By the definition of $c$, for each $x \in X$ we have:
\begin{enumerate}
    \item $\varphi(x)(e) = x(e)$.
    \item $k(\varphi(x)) = k(x)$.
    \item $k(s_l \cdot \varphi(x)) = k(s_l \cdot x)$ for all $l \in \{1, \dots, d\}$.
\end{enumerate}
Applying property (2)  to the  vertex $ s_i \cdot x$, we  have:
\[ k(s_j \cdot \varphi(x)) = k(\varphi(s_i \cdot x))=k(s_i \cdot x). \]
Combining this with property (3), we have:
\[ k(s_j \cdot x) = k(s_i \cdot x). \]

Suppose for contradiction that $j \neq i$. Because the action is free, we have that $1 \le d_{\mathcal{G}}(s_j \cdot x,s_i \cdot x) \le 2$. Consequently, they are adjacent in the $2$-power graph $\mathcal{G}^2$. Since $k$ is a proper coloring of $\mathcal{G}^2$, we must have $k(s_j \cdot x) \neq k(s_i \cdot x)$, a contradiction. 

This forces $j = i$, and we obtain $\varphi(s \cdot x) = s \cdot \varphi(x)$ for all $s \in S$. Thus, $\varphi$ is a  $\Gamma$-equivariant map. 

Finally, combining this with property $(1)$, for each $g\in\Gamma$, we have
\[ x(g)=(g\cdot x)(e)=\varphi(g\cdot x)(e)
=(g\cdot\varphi(x))(e)=\varphi(x)(g). \]
This proves $\varphi(x) = x$. Hence $\varphi$ is the identity map, conclusively proving that $c$ is a valid Borel distinguishing coloring.
\end{proof}

This  simultaneously yields  corollaries for both finite and countably infinite alphabets, explicitly resolving the free group $\mathbb{F}_2$ hypothesis.

\begin{corollary}
Let $m\ge 1, n\ge 2$, and let $\mathcal{G}$ be the Schreier graph induced by the free part of the shift action $\mathbb{F}_m \curvearrowright n^{\mathbb{F}_m}$. Then the Borel distinguishing number of $\mathcal{G}$ is  finite. 
\end{corollary}

Furthermore, by applying this theorem to the countably infinite alphabet $\omega$, we can determine its exact value as follows.

\begin{corollary}
Let $\Gamma$ be any finitely generated amenable group. Let $\mathcal{G}$ be the Schreier graph induced by the free part of the shift action $\Gamma \curvearrowright \omega^\Gamma$. Then its Borel distinguishing number is exactly $D_B(\mathcal{G}) = \omega$.
\end{corollary}

\begin{proof}
For any finite integer $m \ge 2$, the shift space $m^\Gamma$ naturally embeds into $\omega^\Gamma$. By Corollary \ref{embedding_comparison} and Theorem \ref{lowerbound}(1), we have $D_B(\mathcal{G})\ge m$ for arbitrarily large integers $m$; hence $D_B(\mathcal{G}) \ge \omega$. 

By setting $Y = \omega$ in Theorem \ref{generalbound}, we have $D_B(\mathcal{G}) \le \omega \times (d^2+1)^{d+1} = \omega$. Combining both bounds, we obtain exactly $D_B(\mathcal{G}) = \omega$.
\end{proof}

\section{Acknowledgements}
The authors would like to thank Longyun Ding and Su Gao for their guidance and support. We would also like to thank 
Wei Dai, Xiangxi Hu, Yingying Jiang, Feng Li, and Tianhao Wang for many helpful conversations throughout the development of this work. We also thank Xiaohang Zheng for conversations about amenable entropy theory.


\end{document}